\documentclass{amsart}
\usepackage{amsfonts,amssymb,amscd,amsmath,enumerate,verbatim,calc,graphicx}
\newtheorem{theorem}{Theorem}[section]
\newtheorem{lemma}[theorem]{Lemma}

\newtheorem{corollary}[theorem]{Corollary}
\theoremstyle{definition}
\theoremstyle{definitions}

\newtheorem{remark}[theorem]{Remark}

\theoremstyle{notations}

\theoremstyle{remarks}

\newcommand{\T}{\mathrm}

\newcommand{\Z}{\mathbb{Z}}
\newcommand{\fm}{\frak{m}}

\newcommand{\N}{\mathbb{N}}

\newcommand{\lo}{\longrightarrow}

\begin{document}
\author[]
{S. Akbari $^{1,*}$, E. Estaji$^2$ and M.R. Khorsandi$^3$}
\title[]
{On the unit graph of a non-commutative ring}
\thanks{$^*$Corresponding author}
\thanks{E-mail addresses: s\_akbari@sharif.edu, ehsan.estaji@hotmail.com and khorsandi@shahroodut.ac.ir}
\maketitle

\begin{center}
{\it $^1$Department of Mathematics, Sharif University of Technology,\\ P.O. Box 11155-9415, Tehran, Iran
and\\
School of Mathematics, Institute for Research in Fundamental Sciences (IPM),
P.O. Box 19395-5746, Tehran, Iran.\\
$^2$Department of Pure Mathematics, Ferdowsi University of Mashhad,\\
P.O. Box 1159-91775, Mashhad, Iran.\\
$^3$Department of Mathematics, Shahrood University of Technology,\\ P.O. Box 3619995161-316,
Shahrood, Iran
}
\end{center}

\vspace{0.4cm}
\begin{abstract}
Let $R$ be a ring (not necessary commutative) with non-zero identity. The unit graph of $R$, denoted by $G(R)$, is a graph with elements of $R$ as its vertices and two distinct vertices $a$ and $b$ are adjacent if and only if $a+b$ is a unit element of $R$.
It was proved that if $R$ is a commutative ring and $\fm$ is a maximal ideal of $R$ such that $|R/\fm|=2$, then $G(R)$ is a complete bipartite graph if and only if $(R, \fm)$ is a local ring. In this paper we generalize this result by showing that if $R$ is a ring (not necessary commutative), then $G(R)$ is a complete $r$-partite graph if and only if $(R, \fm)$ is a local ring and $r=|R/m|=2^n$, for some $n \in \N$ or $R$ is a finite field. Among other results we show that if $R$ is a left Artinian ring, $2 \in U(R)$ and the clique number of $G(R)$ is finite, then $R$ is a finite ring.
\end{abstract}
\vspace{0.5cm}
\section{Introduction}
One of the interesting and active area in the last decade is using graph theoretical tools to study the algebraic structures. There are several papers are devoted to study of rings in this approach (cf. \cite{A+K+M+M}, \cite{A+B}, \cite{A+L}, \cite{Wu1}, \cite{S+B}, \cite{Wu2} and \cite{Wu3}). The {\it unit graph} of $R$, denoted by $G(R)$, is a graph whose vertices are elements of $R$ and two distinct vertices $a$ and $b$ are adjacent if and only if $a+b$ is a unit element of $R$.
Chung and Grimaldi introduced and investigated the unit graph of $\Z_n$ (the integer modulo $n$) (cf. \cite{C} and \cite{G}). In this article $R$ is a ring (not necessary commutative) with non-zero identity. We denote the Jacobson radical, the set of unit elements of $R$, the set of maximal left ideals of $R$ and the set of $n \times n$ matrices with entries in $R$ by $J(R)$, $U(R)$, $\T{Max}_{l}(R)$ and $M_n(R)$, respectively. A ring $R$ is a {\it local ring}, if $|\T{Max}_{l}(R)|=1$.

Throughout this paper all graphs are simple (with no loop and multiple edges). For a positive integer $r$, a graph is called {\it $r$-partite} if the vertex set admits a partition into $r$ classes such that vertices in the same partition class are not adjacent. An $r$-partite  graph is called {\it complete} if every two vertices in different parts are adjacent.
A {\it clique} of a graph is a complete subgraph. A {\it coclique} (independent set) in a graph is a set of pairwise non-adjacent vertices. A maximum clique is a clique of the largest possible size in a given graph. The clique number $\omega(G)$ of a graph $G$ is the number of vertices in a maximum clique in $G$. The {\it independence number}, $\alpha(G)$, of a graph $G$ is the size of a largest independent set of $G$. A coloring of a graph is a labeling of the vertices with colors such that no two adjacent vertices have the same color. The smallest number of colors needed to color the vertices of a graph $G$ is called its {\it chromatic number}, and denoted by $\chi(G)$.\\

In this paper we show that if $R$ is a ring, then $G(R)$ is a complete $r$-partite graph if and only if $(R, \fm)$ is a local ring and $r=|R/m|=2^n$, for some $n \in \N$ or $R$ is a finite field. Also, we show that if the independence number of $G(R)$ is finite, then either $R$ is finite or a division ring. Finally, we characterize all rings whose unit graphs are bipartite.

\section{Clique Number and Chromatic Number of $G(R)$}
In this section we would like to study some graph theoretical parameters whose finiteness cause the graph $G(R)$ is finite. We start this section with the following lemma.

\begin{lemma}\label{l1}
Let $R$ be a ring. Then the following hold:
\begin{itemize}
\item [(a)] $\omega(G(R/J(R))) \leqslant \omega(G(R))$.
\item [(b)] $\chi(G(R/J(R))) \leqslant \chi(G(R))$.
\item [(c)] If $2 \not \in U(R)$, then $ \chi(G(R/J(R)))=\chi(G(R))$.
\end{itemize}
\end{lemma}
\begin{proof}
\begin{itemize}
\item [(a)] Let $ \{ a_i + J(R) | i \in I \}$ be a clique of $G(R/J(R))$. Then it is easy to check that $ \{ a_i  | i \in I \}$ forms a clique in $G(R)$.
\item [(b)] Suppose that $c:V(G(R)) \lo \{1,2, \ldots, \chi(G(R))\}$ is a coloring of $G(R)$. It is not hard to see that the function $c': V(G(R/J(R))) \lo \{1,2, \ldots, \chi(G(R)) \}$ given by $$c'(a+J(R)):=\T{min} \{ c(x)| x+J(R)=a+J(R) \}$$ is a coloring of $G(R/J(R))$ and so $$\chi(G(R/J(R))) \leqslant \chi(G(R)).$$
\item [(c)] Suppose that $c:V(G(R/J(R))) \lo \{1,2 , \ldots, \chi(G(R/J(R))) \}$ is a coloring of $G(R/J(R))$. Now, define a function $$c':V(G(R)) \lo \{1,2 , \ldots, \chi(G(R/J(R))) \}$$ given by $c'(a)=c(a+J(R))$. We claim that $c'$ is a coloring of $G(R)$. To see this let $a,b \in R$ be two adjacent vertices in $G(R)$ and $c'(a)=c'(b)$. Thus $a+b \in U(R)$. If $a+J(R)=b+J(R)$, then $a-b \in J(R)$. This implies that $2a \in U(R)$, and so $2 \in U(R)$, a contradiction. Hence assume that $a+J(R) \neq b+J(R)$. Since $a+b$ is unit and $c(a+J(R))=c(b+J(R))$ we obtain a contradiction. Therefore $c'$ is a coloring of $G(R)$ and the proof is complete.
\end{itemize}
\end{proof}
Before proving the next result we need the following lemma.
\begin{lemma}\label{l2}
Let $R$ be a left Artinian ring and $R / J(R)$ be finite. Then $R$ is a finite ring.
\end{lemma}
\begin{proof}
Since $R$ is a left Artinian ring, there exists $n \in \mathbb{N}$ such that $(J(R))^n=0$ (cf. \cite[Theorem 4.12]{La}). Now, since $R$ is a left Noetherian ring  $(J(R))^i / (J(R))^{i+1}$ is a finitely generated $R/ J(R) $-module and finiteness of $R/J(R)$ concludes the finiteness of $ (J(R))^i / (J(R))^{i+1}$. Using induction on $i$, one can see that $J(R)$ is finite and so $R$ is a finite ring.
\end{proof}
The following theorem shows that in a left Artinian ring $R$, if the maximum clique of $G(R)$ is finite and $2 \in U(R)$,  then $R$ is finite.
\begin{theorem}\label{t1}
Let $R$ be a left Artinian ring, $2 \in U(R)$ and $ \omega(G(R)) < \infty$. Then $R$ is a finite ring.
\end{theorem}

\begin{proof}
First suppose that $J(R)=0$. By Artin-Wedderburn Theorem there are  natural numbers $n_i$ and division rings $D_i$, for $i=1, \ldots , k$ such that $$R \cong M_{n_1}(D_1) \times \cdots \times M_{n_k}(D_k).$$ Since $\omega (G(R)) < \infty$ and $2 \in U(R)$, we find that $\omega(G(M_{n_i}(D_i)))< \infty$ for $i=1, \ldots , k$. We claim that every $D_i$ is finite. To get a contradiction assume that, $D_i$ is infinite. One can construct an infinite clique using infinite number of scalar matrices. Thus $|R| < \infty$. In the case $J(R) \neq 0$, the assertion follows from Lemma \ref{l1}, Part(a) and Lemma \ref{l2}.
\end{proof}
\begin{remark}
The Artinian property is a necessary condition in Theorem \ref{t1}. To see this we note that $\mathbb{Z}_3[x]$ is a non-Artinian ring with $\omega(\mathbb{Z}_3[x])=2$.
\end{remark}
\begin{theorem}\label{p1}
Let $R$ be a ring such that $\omega(G(R))< \infty$, $|\T{Max}_l(R)|< \infty$ and $2 \in U(R)$. Then $R$ is a finite ring.
\end{theorem}
\begin{proof}
Let $\T{Max}_l(R)= \{\fm_1, \ldots,\fm_n\}$. Then $$R/J(R) \cong R/\fm_1 \times \cdots \times R/\fm_n .$$ By Lemma \ref{l1}, Part(a), $\omega(G(R/J(R)))< \infty$. Hence $\omega(G(R/\fm_i))< \infty$, for $i=1,\ldots, n$. Thus $R/\fm_i$ is finite for $i=1,\ldots, n$ and so $R/J(R)$ is finite. On the other hand, $1+J(R)$ is a clique in $G(R)$ and so $|J(R)|$ is finite. This completes the proof.
\end{proof}
The following remark shows that the finiteness of $\T{Max}_{l}(R)$ and $2 \in U(R)$ are  not superfluous in Theorem \ref{p1}.
\begin{remark}\label{e1}
Let $R_1=\Z_3[x]$ and $R_2=\Z_2[x_1,x_2,\ldots]/({x_1}^2,{x_2}^2,\ldots)$. Then $\omega(G(R_1))=2,\ |\T{Max}_l(R_1)|= \infty$ and $2 \in U(R_1)$. Also, $\omega(G(R_2))=2,\ |\T{Max}_l(R_2)|= 1$ and $2 \not \in U(R_2)$.
\end{remark}
\begin{remark}
Let $R$ be a ring such that $\omega(G(R))< \infty$ and $2 \in U(R)$. In view of the proof of Theorem \ref{p1}, $|J(R)| < \infty$. If $x \in J(R)$, then finiteness of $J(R)$ implies that $x ^i = x^j$ for some $i,j \in \mathbb{N}$ with $i<j$. Thus $x^i=0$ and so $J(R)$ is nilpotent.
\end{remark}
Now, we provide a lower bound for the clique number of unit graph of a ring in terms of the number of maximal ideals.
\begin{theorem}\label{t4}
Let $R$ be a ring such that $|\T{Max}_l(R)|< \infty$ and $2 \in U(R)$. Then $\omega(G(R)) \geqslant |\T{Max}_l(R)|+1 $.
\end{theorem}
\begin{proof}
Let $\T{Max}_l(R)= \{\fm_1, \ldots,\fm_n\}$. Then $$R/J(R) \cong R/\fm_1 \times \cdots \times R/\fm_n .$$ Since $ 2 \in U(R)$,
it is easy to check that the set $$\{(0,1,\ldots,1),(1,0,1,\ldots,1),\ldots,(1,\ldots,1,0),(1,\ldots,1) \}$$ forms a clique in $G(R/\fm_1 \times \cdots \times R/\fm_n)$. Hence $\omega(G(R/J(R)) \geqslant n+1$ and by Lemma \ref{l1}, Part(a), $\omega(G(R))\geqslant n+1$.
\end{proof}
Note that the ring $R_1$ given in Remark \ref{e1} shows that the finiteness of $\T{Max}_l(R)$ in Theorem \ref{t4} is not superfluous.
\\

The next result shows that if the independence number of $G(R)$ is finite, then $R$ is finite or $R$ is a division ring.

\begin{theorem}\label{t2}
Let $R$ be a ring and $ \alpha(G(R)) < \infty$. Then $|R| < \infty$ or $R$ is a division ring.
\end{theorem}
\begin{proof}
Let $I$ be a proper left ideal of $R$. Clearly, $I$ is an independent set and so $|I| < \infty $. Let $0 \neq x \in R$. If $Rx \neq R$, then $|Rx| < \infty$. On the other hand, we have $|\T{Ann}_l(x)| <\infty$. Since $R / \T{Ann}_l(x) \cong Rx$ as an abelian group, thus $|R|< \infty$. Hence if $R$ is infinite, then every $0 \neq x $ is left invertible. Similarly, every $0 \neq x$ is right invertible. This implies that $R$ is a division ring and the proof is complete.
\end{proof}
\begin{corollary}
Let $R$ be a ring. If $3 \leqslant \alpha(G(R)) < \infty$, then $|R| < \infty$.
\end{corollary}
\section{Rings whose unit graphs are complete $r$-partite}
In \cite[Theorem 3.5]{A+M+P+Y}, the authors showed that if $R$ is a commutative ring and $\fm$ is a maximal ideal of $R$ such that $|R/\fm|=2$, then $G(R)$ is a complete bipartite graph if and only if $(R, \fm)$ is a local ring. In the following theorem we generalize their result and characterize all rings (not necessary commutative) whose unit graphs are complete multipartite.
\begin{theorem}\label{t3}
Let $R$ be a ring. Then $G(R)$ is a complete $r$-partite graph if and only if $(R,\fm)$ is a local ring and $r=|R/\fm|=2^n$, for some $n \in \mathbb{N}$ or $R$ is a finite field.
\end{theorem}
\begin{proof}
Let $G(R)$ be a complete $r$-partite graph and $V$ be the part containing zero. Thus $V=R \backslash U(R)$. For every $x \in R$, if $ x \not \in U(R)$, then $x \in V$ and moreover $1-x$ is adjacent to $x$ (Note that $1-x\neq x$, because if $2x=1$, then $x$ is a unit, a contradiction). Thus $1-x \in U(R)$. The above argument shows that for every $x \in R$, either $x \in U(R)$ or $1-x \in U(R)$.
Now, if $|\T{Max}_l(R)| \geqslant 2$, then consider two distinct maximal left ideals $\fm_1,\fm_2$. Then $\fm_1+\fm_2= R$ and $\alpha + \beta =1$, where $\alpha \in \fm_1$ and $\beta \in \fm_2$. Since $\beta = 1- \alpha$ and $ \alpha \not \in U(R)$, $\beta=1-\alpha \in U(R)$ which is a contradiction. So $R$ is a local ring with a unique maximal left ideal $\fm$. First suppose that $2 \in U(R)$. Hence for every $0\neq x \in R$, $x \neq -x$ and so $\{-x,x\}$ are contained in one part of $G(R)$. Without loss of generality assume that $\{-1,1\}\subseteq V_2$, where $V_1=R \backslash U(R),V_2, \ldots, V_r$ are parts of $G(R)$. We claim that $\fm=0$. Let $\fm \neq 0$ and $0 \neq z \in \fm$. Therefore $1+z$ and $-1-z$ are contained in some part $V_j$, $j \neq 1$. Note that $j \neq 2$, because $1+z$ is adjacent to 1. But $-1-z$ and 1 are not adjacent, a contradiction. Thus $\fm=0$. By Theorem 19.1 of \cite{La}, $R$ is a division ring and every part of $G(R)$ is of the form $\{-x,x\}$, for some $x \in R$. Since $G(R)$ is a complete $r$-partite graph, so by Theorem 13.1 of \cite{La}, $R$ is a finite field. Now,  suppose that $2 \in \fm$. If $R / \fm$ is infinite, then $ \omega(R / \fm)$ is infinite. Now, by Lemma \ref{l1}, Part(a), $\omega(R)$ is infinite. Since $G(R)$ is a complete $r$-partite graph, $\omega(G(R))=r$, a contradiction. Therefore $R / \fm$ is finite. Now, since $2 \in \fm$, $\T{char}(R / \fm)=2$ and this implies that $|R / \fm| = 2^n$, for some $ n \in \N$.

Conversely, let $(R,\fm)$ be a local ring with $|R/\fm|=2^n$, for some $n \in \mathbb{N}$. Since $R/\fm$ is a field with $\T{char}(R/\fm)=2$, $G(R/\fm)$ is a complete graph. Also, $2 \not \in U(R)$ implies that each coset $a+\fm$ is an independent set of $G(R)$. It is not hard to see that $G(R)$ is a complete $r$-partite graph that cosets of $R/\fm$ form a partition of $G(R)$. If $R$ is a finite field, then clearly $G(R)$ is a complete multipartite graph.
\end{proof}
%
In the following theorem we characterize all rings whose unit graphs are bipartite.
\begin{theorem} Let $R$ be a ring.
   \begin{itemize}
      \item[(a)] If $J(R)\neq0$ and $2\in U(R)$, then $G(R)$ is not a bipartite graph.
      \item[(b)] If $2\not\in U(R)$, then $G(R)$ is a bipartite graph if and only if $G(R/J(R))$ is a bipartite graph.
      \item[(c)] If $R$ is a semisimple left Artinian ring, then $G(R)$ is a bipartite graph if and only if either $R \cong \Z_3$ or $R$ contains a summand isomorphic to $\Z_2$.
   \end{itemize}
\end{theorem}
\begin{proof}
    \begin{itemize}
      \item[(a)] Let $0 \neq x \in J(R)$. Then 0, 1 and $1-x$ form a cycle.
      \item[(b)] This follows directly from Lemma \ref{l1}, Part(c).
      \item[(c)] If $R \cong \Z_3$, then $G(R)$ is a bipartite graph. Hence we can assume that $R \not \cong \Z_3$. Since $G(\Z_2)$ is a bipartite graph, it is easy to check that if $R$ contains a summand isomorphic to $\Z_2$, then $G(R)$ is a bipartite graph. Suppose that $R$ does not contain a summand isomorphic to $\Z_2$ or $\Z_3$. Now, by Artin-Wedderburn Theorem,  we have $R \cong M_{n_1}(D_1) \times \cdots  \times M_{n_k}(D_k)$, where $n_i \in \N$ and $D_i$ is a division ring, for $i=1,\ldots,k$. Let $S=M_n(D)$, where either $n\geqslant 2$ is a natural number and $D$ is a division ring or $n=1$ and $|D| \geqslant 4$. We show that $G(S)$ contains a triangle. Using block decomposition, it is sufficient to show that the assertion holds for $n=2,3$. If $n=2$,
then $A =\left(
          \begin{array}{cc}
             -1 & 1 \\
              0 & 1 \\
          \end{array}
        \right)$,
     $B=\left(
         \begin{array}{cc}
                   1 & 0 \\
                   1 & 1 \\
           \end{array}
          \right)$
and zero matrix form a cycle. In the case $n=3$,
$A=\left(
 \begin{array}{ccc}
   1 & 0 & 0 \\
    0 & -1 & 1 \\
     0 & 0 & 1 \\
      \end{array}
       \right)$,
 $B=\left(
  \begin{array}{ccc}
     -1 & 1 & 0 \\
       0 & 1 & 0 \\
       1 & 1 & 1 \\
        \end{array}
         \right)$
and zero matrix form a cycle. Moreover, if $n=1$ and $|D| \geqslant 4$, then there exist non-zero elements $x$ and $y$ such that $x+y \neq 0$. Hence $G(S)$ contains a triangle. Thus $G(R)$ contains a triangle and so $G(R)$ is not bipartite.

Now, suppose that $R$ has no summand isomorphic to $\Z_2$ and $R$ contains a summand isomorphic to $\Z_3$. Thus $$R \cong (\Z_3)^l \times M_{n_1}(D_1) \times \cdots \times M_{n_k}(D_k),$$ where either $l \in \N$ and $n_i \geqslant 2$ or $n_i=1$ and $|D_i|\geqslant4$, for $i=1,\ldots,k$. By the above argument for every $i=1,\ldots,k$, there exists a triangle $\alpha_i-\beta_i-\gamma_i-\alpha_i$ in $G(M_{n_i}(D_i))$.
Thus $$(0,\alpha_1,\ldots,\alpha_k)-(1,\beta_1,\ldots,\beta_k)-(1,\gamma_1,\ldots,\gamma_k)-(0,\alpha_1,\ldots,\alpha_k)$$  is a triangle in $G(R)$ which implies that $G(R)$ is not bipartite. This completes the proof.
   \end{itemize}
\end{proof}

\noindent \textbf{Acknowledgement.} The authors are indebted to the School of Mathematics, Institute for Research in Fundamental Sciences (IPM) for support. The research of the first author was in part supported by a grant from IPM (No. 91050212).

\end{document}